\documentclass[11pt,leqno]{amsart}
\usepackage{amsmath,amssymb,amsthm}
\usepackage{hyperref}

\usepackage{bbm}

\renewcommand{\leq}{\leqslant}

\newcommand{\Lip}{{\mathrm{Lip}}_0}

\newtheorem{theorem}{Theorem}[section]
\newtheorem{lemma}[theorem]{Lemma}

\newtheorem{proposition}[theorem]{Proposition}
\newtheorem{corollary}[theorem]{Corollary}
\theoremstyle{definition}

\theoremstyle{remark}
\newtheorem{remark}[theorem]{Remark}

\numberwithin{equation}{section}

\def\fnote#1{\footnote}

\def\ignora#1{}
\def\n3#1{\left\vert  \! \left\vert \! \left\vert \, #1 \, \right\vert \!
  \right\vert \! \right\vert }

\begin{document}

\title[$L_1$-preduals and extensions of Lipschitz maps]{ A characterisation of $L_1$-preduals in terms of extending Lipschitz maps }

\author{ Abraham Rueda Zoca }\thanks{The author was supported by Juan de la Cierva-Formaci\'on fellowship FJC2019-039973, by MTM2017-86182-P (Government of Spain, AEI/FEDER, EU), by MICINN (Spain) Grant PGC2018-093794-B-I00 (MCIU, AEI, FEDER, UE), by Fundaci\'on S\'eneca, ACyT Regi\'on de Murcia grant 20797/PI/18, by Junta de Andaluc\'ia Grant A-FQM-484-UGR18 and by Junta de Andaluc\'ia Grant FQM-0185.}
\address{Universidad de Murcia, Departamento de Matem\'aticas, Campus de Espinardo 30100 Murcia, Spain} \email{ abraham.rueda@um.es}
\urladdr{\url{https://arzenglish.wordpress.com}}



\maketitle

\begin{abstract}
We characterise the Banach spaces $X$ which are $L_1$-predual as those for which every Lipschitz compact mapping $f:N\longrightarrow X$ admits, for every $\varepsilon>0$ and every $M$ containing $N$, a Lipschitz (compact) extension $F:M\longrightarrow X$ so that $\Vert F\Vert\leq (1+\varepsilon)\Vert f\Vert$. Some consequences are derived about $L_1$-preduals and about Lipschitz-free spaces.
\end{abstract}

\section{Introduction}

One of the most celebrated results around Lipschitz functions spaces is the celebrated McShane extension theorem, which asserts that any Lipschitz function $f:N\longrightarrow \mathbb R$ can be extended, for every $M$ containing $N$, to a Lipschitz mapping $F:M\longrightarrow \mathbb R$ without increasing its Lipschitz constant (see e.g. \cite[Theorem 1.33]{weaver}). An analogous result in Banach space theory (which is strongly connected to via \textit{Lipschitz-free spaces}, see \cite[P. 85]{weaver}), which says that given any Banach space $X$, any subspace $M\subseteq X$ and any bounded linear functional $f:M\longrightarrow\mathbb R$ then $f$ can be extended to the whole $X$ in a norm preserving way.

These two nice results, which deal with scalar valued functions, turn out to be false when one considers vector-valued mappings (see \cite[Chapter 2]{beli} for examples). Once one discovers this unpleasant behaviour, we can wonder by those Banach spaces $X$ which are good targets for extending maps.

Concerning the question for bounded linear operators in Banach spaces, the Banach spaces $X$ which give a general Hahn-Banach extension theorem for bounded operators are the \textit{injective Banach spaces} (see e.g. \cite{accgm}). This class of Banach spaces, though being very pleasant for the problem of extending operators, turns out to be very small as reveals, for instance, the fact that there is not any separable injective Banach space. Because of this reason, another possibility is considering Banach spaces which extends a smaller class of bounded operators. In that spirit,  the celebrated result due to J.~Lindenstrauss \cite[Theorem 6.1]{linds} asserts that, given a Banach space $X$, the following are equivalent:
\begin{itemize}
\item Every compact operator $T:Y\longrightarrow X$ has, for every $\varepsilon>0$, a compact extension $\hat T:Z\longrightarrow X$ ($Z\supset Y$) with $\Vert \hat T\Vert\leq (1+\varepsilon)\Vert T\Vert$.
\item $X^*$ is (isometrically) an $L_1(\mu)$ space (for some space $\mu$).
\end{itemize}
Because of this reason these spaces are commonly known as Lindenstrauss spaces (here, however, we will say that $X$ is an \textit{$L_1$-predual}).

Conerning extensions of vector-valued Lipschitz mapping, it has been considered in \cite[Section 2.3]{beli} the concept of \textit{contraction extension property}. Given two Banach spaces $X$ and $Y$, it is said that the pair $(X,Y)$ has the \textit{contraction extension property (CEP)} if, for every $N\subseteq X$ and every Lipschitz function $f:N\longrightarrow Y$ there exists an extension $F:X\longrightarrow Y$ with the same Lipschitz constant. This concept is, again, very restrictive, as proves the fact that if $(X,X)$ has the CEP then, either $X$ is a Hilbert space (if $X$ is strictly convex), or $X$ is an $L_1$-predual if it is not strictly convex \cite[Theorem 2.11]{beli}.

In view of the presence of $L_1$-preduals in connection with the CEP, we can wonder how can $L_1$-preduals be characterised in terms of extensions of Lipschitz mappings. This is the aim of this note (see Theorem \ref{theo:caraL1predual}). We prove that, given a Banach space $X$, the following are equivalent:
\begin{itemize}
\item For every pair of metric spaces $N\subseteq M$ and every Lipschitz compact mapping $f:N\longrightarrow X$ there exists, for every $\varepsilon>0$, a Lipschitz compact mapping $F:M\longrightarrow X$ so that $\Vert F\Vert\leq (1+\varepsilon)\Vert f\Vert$.
\item $X$ is an $L_1$-predual.
\end{itemize}
See Section \ref{section:preliminar} for the notions of norm of a Lipschitz mapping and for the definition of Lipschitz compact maps. First, we prove in Theorem \ref{theo:condisufi} that if $X$ is an $L_1$-predual then we get the desired extension property. For this, we need the \textit{Lipschitz-free space} together with the above mentioned celebrated characterisation of $L_1$-preduals. For the converse, we will make use of a characterisation of $L_1$-preduals in terms of a property of intersections of balls. From the main characterisation we get two consequences. The first one is that, in order to prove that a Banach space $X$ is an $L_1$-predual, it is enough to check the property of extension of bounded linear operators described above \cite[Theorem 6.1(3)]{linds} in the class of Lipschitz-free spaces. On the other hand, we get that if $X$ and $Y$ are two Banach spaces for which, for every $\varepsilon>0$, there exists an onto bi-Lipschitz isomorphism $\phi_\varepsilon: X\longrightarrow Y$ so that $\Vert \phi_\varepsilon\Vert\Vert \phi_\varepsilon^{-1}\Vert\leq 1+\varepsilon$, then $X$ is an $L_1$-predual if, and only if, $Y$ is an $L_1$-predual.

\section{Preliminary results}\label{section:preliminar}

Given a metric space $(M,d)$, we will denote by $B(x,r)$ closed ball centred at $x\in M$ with radius $r>0$. Given a Banach space $X$ and a Lipschitz function $f:M\longrightarrow X$, we will denote by $\Vert f\Vert$ to the best Lipschitz constant of $f$, i.e., to the quantity
$$ \Vert f \Vert := \sup\left\{\frac{\Vert f(x)-f(y)\Vert }{d(x,y)} \colon x,y\in M,\, x \neq y \right\}.$$
In general, the previous expression does not define a norm in the vector space of all the Lipschitz functions since every constant function has $0$-Lipschitz constant. However,  we keep this notation because this expression is ``close'' to be a norm in the following sense.

Consider a metric space $M$ in which we distinguish an element, called $0$. Given a Banach space $X$, we write $\Lip(M,X)$ ($\Lip(M)$ when $X=\mathbb R$) to denote the vector space of all Lipschitz maps $f: M\longrightarrow X$ which vanish at $0$. When endowed with $\Vert \cdot\Vert$ defined as before, $\Lip(M,X)$ turns out to be a Banach space and, in the case when $X=\mathbb R$, $\Lip(M)$ is even a dual Banach space.

We denote by $\delta$ the canonical isometric embedding of $M$ into $\Lip(M)^*$, which is given by $\langle f, \delta(x) \rangle =f(x)$ for $x \in M$ and $f \in \Lip(M)$. We denote by $\mathcal{F}(M)$ the norm-closed linear span of $\delta(M)$ in the dual space $\Lip(M)^*$, which is usually called the \textit{Lipschitz-free space over $M$}; for background on this, see the survey \cite{godesurv} and the book \cite{weaver} (where it receives the name of ``Arens-Eells space''). It is well known that $\mathcal{F}(M)$ is an isometric predual of the space $\Lip(M)$ \cite[p. 91]{godesurv}. We will write $\delta_x:=\delta(x)$ for $x\in M$.

A fundamental result in the theory of Lipschitz-free spaces is that, roughly speaking, Lipschitz-free spaces linearise Lipschitz maps. In a more precise language, given a metric space $M$, a Banach space $X$ and a Lipschitz map $f:M\longrightarrow X$ such that $f(0)=0$, there exists a bounded operator $T_f:\mathcal F(M)\longrightarrow X$ such that $\Vert T_f\Vert=\Vert f\Vert_L$ defined by
$$T_f(\delta_m):=f(m) \ , \quad m\in M.$$
Moreover, the mapping $f\longmapsto T_f$ is an onto linear isometry between $\Lip(M,X)$ and the space of bounded operators $\mathcal L(\mathcal F(M),X)$.
This linearisation property makes Lipschitz-free spaces a precious magnifying glass to study Lipschitz maps between metric spaces, and for example it relates some well-known open problems in the Banach space theory to some open problems about Lipschitz-free spaces (see \cite{godesurv}). 

When dealing with the celebrated characterisation of $L_1$-preduals given in \cite[Theorem 6.1]{linds} it is necessary to deal with compact operators. We will consider the following concept of Lipschitz compact functions, already considered in \cite{jsv}. Given a metric space $M$, a Banach space $X$ and a Lipschitz mapping $f:M\longrightarrow X$, we say that $f$ is \textit{compact} if the set
$$\left\{\frac{f(x)-f(y)}{d(x,y)}: x,y\in M, x\neq y \right\}\subseteq X$$
is relatively compact. It is inmediate that if $M$ is finite or if $X$ is finite dimensional, then every Lipschitz mapping $f:M\longrightarrow X$ is compact. This fact becomes even clearer with the fact that, given a Lipschitz mapping $f:M\longrightarrow X$ which vanishes at any distinguised point $0\in M$, it follows that $f$ is compact if, and only if, its linearisation operator $T_f:\mathcal F(M)\longrightarrow X$ is compact (i.e $T(B_{\mathcal F(M)})\subseteq X$ is relatively compact) \cite[Proposition 2.1]{jsv}.

In the proof of the main result the following fact is essential: given a Banach space $X$, if $X^{**}$ is an $L_1$-predual, then so is $X$. Even though this result is well known (see e.g. \cite[Lemma 2.12]{beli}), let us present a justification from a wider perspective. To do so, let us introduce the concept of \textit{ideals}. Let $Z$ be a subspace of a Banach space $X$.
We say that $Z$ is \textit{locally complemented} (a.k.a. an ideal) in $X$ if, for each $\varepsilon>0$ and for each finite-dimensional subspace $E\subseteq X$, there exists a linear operator $T:E\to Z$ satisfying
\begin{enumerate}
\item\label{item:ai-1}
  $T(e)=e$ for each $e\in E\cap Z$, and
\item\label{item:ai-2}
  $\Vert T(e)\Vert\leq  (1+\varepsilon) \Vert e \Vert$
  for each $e\in E$,
\end{enumerate}
i.e., $T$ is a $(1+\varepsilon)$ operator fixing the elements of $E$. See \cite{gks,kalton84} for background. Note that the Principle of Local Reflexivity \cite[Theorem 9.15]{fab} implies that $X$ is an ideal in $X^{**}$ for every Banach space $X$.

In general, the property of being an $L_1$-predual is stable under taking locally complemented subspaces. This fact was implicitly proved in \cite[Theorem 1]{rao2} using a characterisation of $L_1$-preduals in terms of a property of intersections of balls. However, let us include here a proof, for the sake of completeness, making use of extension of compact operators.

\begin{lemma}\label{inheriL1predual}
Let $X$ be an $L_1$ predual and let $Y$ be an ideal in $X$. Then $Y$ is an $L_1$ predual.
\end{lemma}

\begin{proof}
In order to prove that $Y$ is an $L_1$ predual we will make use of the equivalence between (2) and (4) in \cite[Theorem 6.1]{linds}. To this end, consider a Banach space $Z$ such that $\dim(Z)\leq 3$, a bounded operator $T:Z\longrightarrow Y$, a positive $\varepsilon$ and another Banach space $V$ containing $Z$ such that $dim\left(V/Z\right)=1$, and let us prove that there exists a bounded operator $\Psi:V\longrightarrow Y$ extending $T$ such that $\Vert \Psi\Vert\leq (1+\varepsilon)\Vert T\Vert$. To this end consider a positive $\delta$ such that $(1+\delta)^2<1+\varepsilon$ and consider $i:Y\longrightarrow X$ the inclusion operator. Now, since $X$ is an $L_1$ predual, there exists an extension $\Phi:V\longrightarrow X$ of $i\circ T$ such that $\Vert \Phi\Vert\leq (1+\delta)\Vert i\circ T\Vert\leq (1+\delta)\Vert T\Vert$. Now, since $\Phi(V)$ is a finite-dimensional subspace of $X$ and $Y$ is an ideal in $X$, we can find a linear operator $P:\Phi(V)\longrightarrow Y$ satisfying $\Vert P\Vert\leq 1+\delta$ and such that
$$P(z)=z\ \forall z\in \Phi(V)\cap Y.$$
Finally define $\Psi=P\circ \Phi:V\longrightarrow Y$. It is obvious that $\Vert \Psi\Vert\leq (1+\delta)^2\Vert T\Vert\leq (1+\varepsilon)\Vert T\Vert$. Moreover, given $z\in Z$ one has
$$\Psi(z)=P(\Phi(z))=P(i(T(z)))=T(z),$$
since $\Phi(z)\in Y$. Consequently, $\Psi$ satisfies the desired requirements, so we are done.
\end{proof}

\section{Main results}

Let us begin with the following result, where we will prove a result of extension of Lipschitz mappings throughout the characterisation of $L_1$-predual spaces given in \cite[Theorem 6.1]{linds}, for which we will make use of Lipschitz-free spaces.

\begin{theorem}\label{theo:condisufi}
Let $X$ be an $L_1$-predual space. Then, for every pair of metric spaces $N\subseteq M$ and every Lipschitz compact function $f:N\longrightarrow X$ there exists, for every $\varepsilon>0$, a Lipschitz compact function $F:M\longrightarrow X$ so that
\begin{enumerate}
\item $F(x)=f(x)$ holds for every $x\in N$ and,
\item $\Vert F\Vert\leq (1+\varepsilon)\Vert f\Vert$.
\end{enumerate}
\end{theorem}

\begin{proof}
Select a distinguised point in $N$ which we will call $0$ for convenience. Assume first that $f(0)=0\in X$.

Pick $T_f:\mathcal F(N)\longrightarrow X$ the linearisation operator of $f$. Since $X$ is an $L_1$-predual then, given $\varepsilon>0$, there exists an extension $T:\mathcal F(M)\longrightarrow X$ so that $\Vert T\Vert\leq (1+\varepsilon)\Vert T_f\Vert=(1+\varepsilon)\Vert f\Vert$ \cite[Theorem 6.1]{linds}. Now $F:M\longrightarrow X$ defined by $F(x):=T(\delta_x)$ for every $x\in M$ does the trick.

For the general case, if $f(0)\neq 0$, define the auxiliar mapping $g(x)=f(x)-f(0)$ for every $x\in N$. Since $g(0)=0$, we can find $G:M\longrightarrow X$ extending $g$ so that $\Vert G\Vert\leq (1+\varepsilon)\Vert g\Vert=(1+\varepsilon)\Vert f\Vert$. Now $F:M\longrightarrow X$ defined by $F(x)=G(x)+f(0)$ is the desired extension.\end{proof}

Our aim will be to establish a converse. To save a bit of notation in the intermediate results, let us say that a Banach space $X$ has the \textit{$\varepsilon$-finite extension property} if, for every pair of finite metric spaces $N$ and $M$ with $N\subseteq M$ and for every Lipschitz function $f:N\longrightarrow X$ there exists an extension $F:M\longrightarrow X$ so that $\Vert F\Vert\leq (1+\varepsilon)\Vert f\Vert$. We will talk of the \textit{$0$-finite extension property} if $\varepsilon=0$ can be taken in the previous definition.

The following result says that if a Banach space $X$ has the $\varepsilon$-finite extension property then, going to the bidual space $X^{**}$, one can take $\varepsilon=0$. This can be seen as an analogous counterpart to what happen in the case of bounded linear operators (see Remark \ref{remark:linealbidual}).

\begin{proposition}\label{prop:bidual}
Let $X$ be a Banach space with the $\varepsilon$-finite extension property. Then $X^{**}$ has the $0$-finite extension property.
\end{proposition}

\begin{proof} In the proof we inspire in a Lindenstrauss compactness argument (see e.g. \cite[Lemma III.4.3]{hww}). Let $M$ be a finite metric space, $N\subseteq M$ and $f:N\longrightarrow X^{**}$ be a Lipschitz function. Let us find a norm-preserving extension. To this end, define
$$\mathcal F:=\{(\varepsilon,F): \varepsilon>0, F\subseteq X^*\mbox{ finite-dimensional subspace}\}.$$
$\mathcal F$ is a directed set with the order $(\varepsilon,F)\leq (\delta, E)$ if $\delta\leq \varepsilon$ and $F\subseteq E$.

Pick $V:=span\{f(N)\}$. For every $(\varepsilon, F)\in \mathcal F$ we find, by the Principle of Local Reflexivity \cite[Lemma 9.15]{fab}, a bounded operator $T_{(\varepsilon,F)}\longrightarrow X$ so that
\begin{enumerate}
\item $\Vert T_{(\varepsilon,F)}\Vert\leq 1+\varepsilon$ and,
\item $x^*(T(v))=v(x^*)$ for every $x^*\in F$ and every $v\in V$.
\end{enumerate}
Consider $f_{(\varepsilon,F)}=T_{(\varepsilon,F)}\circ f: N\longrightarrow X$. By the linearity of $T_{(\varepsilon,f)}$ it is inmediate that $\Vert f_{(\varepsilon,F)}\Vert\leq (1+\varepsilon)\Vert f\Vert$. Since $X$ has the $\varepsilon$-finite extension property there exists a Lipschitz mapping $G_{(\varepsilon,F)}: M\longrightarrow X$ so that $G_{(\varepsilon,F)|N}=f$ and $\Vert G_{(\varepsilon, F)}\Vert\leq (1+\varepsilon)^2\Vert f\Vert$. Note that, by the construction, there exists $K$ large enough so that $G_{(\varepsilon,F)}(M)\subseteq K B_{X^{**}}$ and, consequently, 
$$(G_{(\varepsilon, F)})_{(\varepsilon, F)\in \mathcal F}\subseteq (K B_{X^{**}},w^*)^M,$$
which is a compact space under the product topology. Consequently, pick a cluster point $G$ of the net $G_{(\varepsilon, F)}$. Let us prove that $G$ satisfies our purposes. First of all, $\Vert G\Vert\leq \Vert f\Vert$. Suppose, on the contrary, that $\Vert G\Vert>\Vert f\Vert$. This means that there are $x\neq y\in M$ and $\delta>0$ so that
$$\frac{\Vert G(x)-G(y)\Vert}{d(x,y)}>(1+\delta)\Vert f\Vert,$$
and so find $x^*\in B_{X^*}$ so that $\frac{x^*(G(x))-x^*(G(y))}{d(x,y)}>(1+\delta)\Vert f\Vert$. Since $G$ is a cluster point of $G(\varepsilon,F)$ find $(\varepsilon,F)\in \mathcal F$ such that $(1+\varepsilon)^2<1+\delta$ and $x^*\in F$ so that $\frac{x^*(G_{(\varepsilon,F)}(x))-x^*(G_{(\varepsilon,F)}(y))}{d(x,y)}>(1+\delta)\Vert f\Vert$.  Since $x^*\in B_{X^*}$ we get that
$$(1+\delta)\Vert f\Vert\leq x^*\left(\frac{G_{(\varepsilon,F)}(x)-G_{(\varepsilon,F)}(y)}{d(x,y)}\right)\leq \Vert G_{(\varepsilon,F)}\Vert\leq (1+\varepsilon)^2\Vert f\Vert,$$
which entails a contradiction. Consequently, $\Vert G\Vert\leq \Vert f\Vert$.

We will derive actually the equality of norms when we have proved that $G$ coincides with $f$ on $N$. To this end we argue again by contradiction. Assume that there exists $x\in N$ so that $G(x)\neq f(x)$. Then we can find $\delta>0$ and $x^*\in B_{X^*}$ so that $x^*(G(x))-x^*(f(x))>\delta$. Since $G$ is a cluster point of $(G_{(\varepsilon,F)})_{(\varepsilon,F)\in\mathcal F}$ find $(\varepsilon,F)\in\mathcal F$ so that $x^*\in F$ and $x^*(G_{(\varepsilon,F)}(x))-x^*(f(x))>\delta$. Since $x\in N$ we get that $G_{(\varepsilon,F)}(x)=F_{(\varepsilon,F)}(x)=T_{(\varepsilon,F)}(f(x))$ and, since $x^*\in F$ and $f(x)\in V$, we derive that
$$x^*(G_{(\varepsilon,F)}(x))=x^*(T_{(\varepsilon,F)}(f(x)))=f(x)(x^*)=x^*(f(x)),$$
so $\delta<x^*(G_{(\varepsilon,F)}(x))-x^*(f(x))=x^*(f(x))-x^*(f(x))=0$, a contradiction. This contradiction proves that $G$ equals $f$ on $N$ and the proof is finished. \end{proof}

\begin{remark}\label{remark:linealbidual}
It is easy to prove from \cite[Theorem 6.1]{linds} that, given a Banach space $X$, then $X$ is an $L_1$-predual if, and only if, for every bounded operator $T:Y\longrightarrow X^{**}$ there exists, for every Banach space $Z$ containing $Y$, a norm-preserving extension $\hat T:Z\longrightarrow X^{**}$. 

Indeed, if $X$ is an $L_1$-predual so is $X^{**}$. If $T:Y\longrightarrow X^{**}$ and $Y$ is contained in $Z$, by \cite[Theorem 6.1 (5)]{linds} there is an extension $G:Z\longrightarrow X^{****}$ with $\Vert G\Vert=\Vert T\Vert$. Now, since there is a contractive projection $\pi:X^{****}\longrightarrow X^{**}$ by a classical result of Dixmier (see e.g. \cite[Exercise 5.7]{fab}) then $\hat T:=\pi\circ G:Z\longrightarrow X^{**}$ yields the desired extension.

The reverse implication is inmediate because the second statement implies that $X^{**}$ is an $L_1$-predual.
\end{remark}

Now we are ready to prove the second central result of this note.

\begin{theorem}\label{theo:condinece}
Let $X$ be a Banach space with the $\varepsilon$-finite extension property. Then $X$ is an $L_1$-predual.
\end{theorem}

\begin{proof}
Let us prove that $X^{**}$ is an $L_1$-predual for which we will prove, thanks to \cite[Theorem 6.1 (12)]{linds}, that given four balls $\{B(x_i,r_i): 1\leq i\leq 4\}\subseteq X^{**}$ with the property that $B(x_i,r_i)\cap B(x_j,r_j)\neq \emptyset$ for every $i\neq j$ then $\bigcap\limits_{i=1}^4 B(x_i,r_i)\neq \emptyset$. To this end, take four such balls.  Now we can find a linear isometry $\phi: X^{**}\longrightarrow \ell_\infty(\Gamma)$ (see e.g. \cite[Proposition 5.11]{fab} and the paragraph below the proof). Since $\phi$ is an isometry then $B(\phi(x_i),r_i)\cap B(\phi(x_j),r_j)\neq \emptyset$ for every $i\neq j$. Since $\ell_\infty(\Gamma)$ is an $L_1$-predual, \cite[Theorem 6.1 (12)]{linds} implies that there exists $z\in \bigcap\limits_{i=1}^4 B(\phi(x_i),r_i)$. Define $f:\{\phi(x_i): 1\leq i\leq 4\}\longrightarrow X^{**}$ by the equation
$$f(\phi(x_i))=x_i.$$
Since $\phi$ is isometric we get that $\Vert f\Vert\leq 1$. By Proposition \ref{prop:bidual} we can extend $f$ to a $1$-Lipschitz map $F$ defined on $\{\phi(x_i):1\leq i\leq 4\}\cup\{z\}$. We claim that $F(z)\in \bigcap\limits_{i=1}^4 B(x_i,r_i)$. Indeed, given $i\in\{1,\ldots, 4\}$, since $z\in B(\phi(x_i),r_i)$ then $\Vert z-\phi(x_i)\Vert\leq r_i$. Hence
$$\Vert F(z)-x_i\Vert=\Vert F(z)-F(\phi(x_i))\Vert\leq \Vert F\Vert \Vert z-\phi(x_i)\Vert\leq r_i$$
which, in other words, means that $F(z)\in B(x_i,r_i)$. Since $i$ was arbitrary we get that $F(z)\in \bigcap\limits_{i=1}^4 B(x_i,r_i)$. 

By \cite[Theorem 6.1 (12)]{linds}, $X^{**}$ is an $L_1$-predual. Consequently, $X$ is an $L_1$-predual, as requested.
\end{proof}

As a consequence we get the following characterisation of $L_1$-predual spaces.

\begin{theorem}\label{theo:caraL1predual}
Let $X$ be a Banach space. The following assertions are equivalent:
\begin{itemize}
\item[(1)] $X$ is an $L_1$-predual.

\item[(2)] For every pair of metric spaces $N\subseteq M$ and every Lipschitz compact function $f:N\longrightarrow X$ there exists, for every $\varepsilon>0$, a Lipschitz compact function $F:M\longrightarrow X$ extending $f$ so that $\Vert F\Vert\leq (1+\varepsilon)\Vert f\Vert$.

\item[(2')] For every pair of metric spaces $N\subseteq M$ and every Lipschitz compact function $f:N\longrightarrow X$ there exists, for every $\varepsilon>0$, a Lipschitz function $F:M\longrightarrow X$ extending $f$ so that $\Vert F\Vert\leq (1+\varepsilon)\Vert f\Vert$.

\item[(3)] For every pair of metric spaces $N\subseteq M$ and every compact operator $T:\mathcal F(N)\longrightarrow X$ there exists, for every $\varepsilon>0$, a compact operator $\hat T:\mathcal F(M)\longrightarrow X$ extending $T$ and such that $\Vert \hat T\Vert\leq (1+\varepsilon)\Vert T\Vert$.

\item[(4)] For every pair of finite metric spaces $N\subseteq M$ and every Lipschitz function $f:N\longrightarrow X$ there exists, for every $\varepsilon>0$, a Lipschitz function $F:M\longrightarrow X$ extending $f$ so that $\Vert F\Vert\leq (1+\varepsilon)\Vert f\Vert$.

\item[(5)] For every pair of finite metric spaces $N\subseteq M$ and every Lipschitz function $f:N\longrightarrow X^{**}$ there exists a norm-preserving extension $F:M\longrightarrow X^{**}$.
\item[(6)] For every pair of finite metric spaces $N\subseteq M$, where $N$ has four points and $M$ has five points, and every Lipschitz function $f:N\longrightarrow X^{**}$ there exists a norm-preserving extension $F:M\longrightarrow X^{**}$.
\end{itemize}
\end{theorem}

\begin{proof}
(1)$\Rightarrow$(2) is Theorem \ref{theo:condisufi}. (2)$\Leftrightarrow$(3) is inmediate and codified in the proof of Theorem \ref{theo:condisufi}. (2)$\Rightarrow$(2')$\Rightarrow$(4) is inmediate, (4)$\Rightarrow$(5) is Proposition \ref{prop:bidual}, and (5)$\Rightarrow$(6) is obvious and (6)$\Rightarrow$(1) is the proof of Theorem \ref{theo:condinece}.
\end{proof}

Some remarks are convenient.

\begin{remark}
\begin{enumerate}
\item In \cite[Theorem 6.1]{linds} it is proved that a Banach space $X$ is an $L_1$-predual if, and only if, for every compact operator $T:Y\longrightarrow X$ there exists, for every $\varepsilon>0$ and for every $Z\supset Y$, a compact extension $\hat T:Z\longrightarrow X$ so that $\Vert \hat T\Vert\leq (1+\varepsilon)\Vert T\Vert$. Theorem \ref{theo:caraL1predual} reveals that it is necessary and sufficient that the above mentioned extension result holds in  for $Y$ and $Z$ in the class of Lipschitz-free spaces.

\item Point (6) in Theorem \ref{theo:caraL1predual} reveals that, in order to get that a Banach space $X$ is an $L_1$-predual, it is enough to prove that every Lipschitz mapping $f:N\longrightarrow X^{**}$ with $N$ having for point has, for every $M$ containing $N$ and just one more point, a norm preserving extension $F:M\longrightarrow X^{**}$. This should be compared with \cite[Theorem 6.1]{linds} where it is proved that $X$ is an $L_1$-predual if, and only if, $X$ extends (for every $\varepsilon>0$) every bounded operator $T:Y\longrightarrow X$ with $dim(Y)\leq 3$ to a bounded operator $\hat T:Z\longrightarrow X$ with $dim(Z/Y)=1$ and $\Vert \hat T\Vert\leq (1+\varepsilon)\Vert T\Vert$.

\item Theorem \ref{theo:caraL1predual} gives a different proof of a fact which is in turn a manifestation of the four ball intersection property: the class of $L_1$-preduals is stable under (non necessarily linear) $(1+\varepsilon)$-isometries in the following sense.
\end{enumerate}
\end{remark}

\begin{corollary}
Let $X$ and $Y$ be two Banach spaces satisfying that, for every $\varepsilon>0$, there exists a bi-Lipschitz mapping $\phi_\varepsilon: X\longrightarrow Y$ so that $\Vert \phi_\varepsilon\Vert\Vert \phi_\varepsilon^{-1}\Vert\leq 1+\varepsilon$. If $X$ is an $L_1$-predual, then $Y$ is an $L_1$-predual.
\end{corollary}

The previous result implies, in particular, that the property of being an $L_1$-predual is preserved by Banach Mazur distance equal to $1$.

\begin{proof}
Let us prove, by Theorem \ref{theo:caraL1predual}, that given two finite metric spaces $N\subseteq M$ and given a Lipschitz mapping $f:N\longrightarrow Y$ there exists, for every $\varepsilon>0$, a Lipschitz map $F:M\longrightarrow Y$ which extends $f$ and so that $\Vert F\Vert\leq (1+\varepsilon)\Vert f\Vert$. To this end, given $\varepsilon>0$, consider a bi-Lipschitz mapping $\phi_\varepsilon:X\longrightarrow Y$ as in the statement. Now, we can apply Theorem \ref{theo:caraL1predual} to the Lipschitz mapping $\phi_\varepsilon^{-1}\circ f: N\longrightarrow X$ and we get, since $X$ is an $L_1$-predual, an extension $G:M\longrightarrow X$  with $\Vert G\Vert\leq (1+\varepsilon)\Vert \phi_\varepsilon^{-1}\circ f\Vert\leq (1+\varepsilon)\Vert \phi_\varepsilon^{-1}\Vert\Vert f\Vert$. Now consider $F:=\phi_\varepsilon\circ G: M\longrightarrow Y$, which is a Lipschitz map of norm $\Vert F\Vert\leq (1+\varepsilon)\Vert \phi_\varepsilon\Vert \Vert \phi_\varepsilon^{-1}\Vert\Vert f\Vert\leq (1+\varepsilon)^2 \Vert f\Vert$. Moreover, it is inmediate to prove that $F(x)=f(x)$ holds for every $x\in N$, as desired.
\end{proof}

\textbf{Acknowledgements:} The author is grateful to Luis Garc\'ia-Lirola, Gilles Godefroy, Gin\'es L\'opez-P\'erez and Miguel Mart\'in for answering his inquires and for valuable suggestions and comments. The research of this paper was suggested by Gilles Godefroy during the PhD defense of Rafael Chiclana. I would take the opportunity to congratulate Rafael Chiclana for his nice thesis and denfese.

\end{document}